\documentclass[12pt, reqno]{amsart}
\usepackage{amsmath, amsthm, amscd, amsfonts, amssymb, graphicx, color}
\usepackage[bookmarksnumbered, colorlinks, plainpages]{hyperref}

\newtheorem{theorem}{Theorem}[section]

\newtheorem{proposition}[theorem]{Proposition}
\newtheorem{corollary}[theorem]{Corollary}
\theoremstyle{definition}
\newtheorem{definition}[theorem]{Definition}
\newtheorem{example}[theorem]{Example}

\theoremstyle{remark}
\newtheorem{remark}[theorem]{Remark}
\numberwithin{equation}{section}
\begin{document}
\begin{center}
\large{\textbf{Characterization of Einstein Poisson warped product space}}
\end{center}
\begin{center}
	{Buddhadev~~Pal and Pankaj Kumar\footnote{The Second author is supported by UGC JRF of India, Ref. No: 1269/(SC)(CSIR-UGC NET DEC. 2016).}}
\end{center}
\vskip 0.2cm
\begin{center}
	Department of Mathematics\\
	Institute of Science\\
	Banaras Hindu University\\
	Varanasi-221005, India.\\
	E-mail: pal.buddha@gmail.com\\
	E-mail: pankaj.kumar14@bhu.ac.in\\
\end{center}
\vskip 0.5cm
\begin{center}
	\textbf{Abstract}
\end{center}
 In this article, we study the problem of the existence and nonexistence of warping function associated with constant scalar curvature on pseudo-Riemannian Poisson warped product space under the assumption that fiber space has constant scalar curvature. We characterize the warping function on Einstein Poisson warped space by taking the various dimensions of base space $B$ (i.e; (1). $dim B=1,$ (2). $dimB\geq2$).\\\\
 \textbf{2010 Mathematics Subject Classification:} 53C15, 53D17.\\\\
\textbf{Key words and phrases:} Einstein manifold, Warped product, Levi-Civita contravariant connection, Poisson structure, pseudo-Riemannian Poisson manifold.
\section{\textbf{Introduction}}
Poisson geometry is developed to provide a powerful technique in symplectic geometry like groupoid/algebroid theory or singularity theory. Poisson manifold considered a phase space in Hamiltonian dynamics.
 Moreover, it is also studied in various branches of mathematics like representation theory, integrable system, non-commutative geometry, quantum groups, etc.\par In \cite{ivs,jpn}, authors studied the Poisson structure on manifolds and provided various information on Poisson manifolds. To the search of covariant connection type notion on Poisson manifold $(M,\Pi)$ associated with a Poisson structure $\Pi$, I. Vaisman \cite{ivs} provided the notion of contravariant connection $\mathcal{D}$. In \cite{rfl}, R. L. Fernandes provided the importance of contravariant connection to study the global properties of Poisson manifold. In \cite{mbo}, M. Boucetta proved the existence of Levi-Civita contravariant connection on Poisson manifold associated with pseudo-Riemannian metric and also introduced the concept of compatibility between a Poisson structure and pseudo-Riemannian metric. In \cite{mbo2}, author introduced the notion of pseudo-Riemannian Poisson manifold by using the concept of compatibility. In \cite{zs}, authors introduced a contravariant analogue of the Laplace operator acting on differential forms.\par
A pseudo-Riemannian manifold ($M$,$\tilde{g}$) of dimension $n>2$ is said to be an Einstein manifold if its Ricci tensor $Ric$ is proportional to the metric $\tilde{g}$ i.e., for some constant $\lambda$ on $M$ $$Ric(X,Y) = \lambda\tilde{g}(X,Y),$$ for every $X,Y\in\mathfrak{X}(M)$. If $\lambda=0,$ then $M$ is called Ricci-flat.\par The notion of warped product manifold was first introduced by Bishop and O'Neill in \cite{lota}, for studying manifolds with negative curvature. It generalizes that of a surface of revolution and has useful application in general theory of relativity. In \cite{jkp,jke}, authors observed that some well-known solutions to Einstein's field equations known in terms of warped products besides they obtain more solutions to Einstein's field equations by using the methods of warped products.\par
Let $(B,\tilde{g_{B}})$ and $(F,\tilde{g_{F}})$ be two pseudo-Riemannian manifolds with dim$(B)=n>0$, dim$(F)=m>0$ and $f$ is a positive smooth function on $B$ then the warped product $M=B \times_f F$ is the product manifold  $B\times F$ furnished with the metric tensor $\tilde{g}^f=\tilde{g_{B}}+f^2\tilde{g_{F}}$ defined for pairs of vector fields $X$, $Y$ on $M=B \times_f F$ by
\begin{equation*}
\tilde{g}^f(X,Y)=\tilde{g_{B}}(\pi_*X,\pi_*Y)+f^2(\pi(.))\tilde{g_{B}}(\sigma_*X,\sigma_*Y),
\end{equation*}
where $\pi$ and $\sigma$ are canonical projections over $B$ and $F$ respectively. $(B,\tilde{g_{B}})$ is known as base space, $(F,\tilde{g_{F}})$ is known as fiber space and $f$ is known as warping function.\par
In \cite{rnm}, authors provided the Poisson structure on product Riemannian manifold and also discussed all geometric properties of product Riemannian Poisson manifold. In \cite{yar}, authors generalized the concept of the preceding statement for warped product of pseudo-Riemannian manifolds equipped with a warped Poisson structure and computed the corresponding contravariant Levi-Civita connection and curvatures. In \cite{sba}, authors have computed the warping functions for a Ricci flat Einstein multiply warped product spaces $M$ with a quarter symmetric connection.\\\par
Motivated from all the above papers we organized our work as follows: In section $2$, first we recall some definitions and relations on Poisson manifold $(M,\Pi)$ associated with contravariant connection $\mathcal{D}$ and remind the existence and uniqueness of contravariant Levi-Civita connection associated with the pair $(\Pi,g^f)$ on $M$, later we provide the relation between contravariant Laplacian operator $\Delta^{\mathcal{D}}$ and Hessian of a smooth function $f$ with respect to contravariant Levi-Civita connection $\mathcal{D}$ associated with the pair $(\Pi,g^f)$ on $M$. Finally, we define the cometric $g^f$ of the warped metric $\tilde{g}^f$ on space $B\times_f F$ associated with bivector $\Pi=\Pi_1+\Pi_2$ and mention Theorem 2.3 for a contravariant warped product of two pseudo-Riemannian Poisson manifolds. In section $3$, we discuss contravariant Levi-Civita connection $\mathcal{D}$, curvature $\mathcal{R}$, Ricci curvature $Ric$ and scalar curvature $S$ on contravariant warped product space $(M=B\times_f F,g^f)$, also for pseudo-Riemannian Poisson warped product space. After that in Theorem 3.7, we prove the existence of constant scalar curvature on pseudo-Riemannian Poisson warped product space $(M=B\times_f F,\Pi,g^f)$ under the assumption that fiber space $F$ has non-zero constant scalar curvature. In section 4, in Theorem 4.1 and Theorem 4.3, we introduce the notion of contravariant Einstein warped product space and Einstein Poisson warped product space respectively. Later we provide Corollary 4.5 which inform that, pseudo-Riemannian Poisson warped product space $(M=B\times_f F,g^f,\Pi)$ is Ricci-flat if and only if $B$ and $F$ are, also verify this by an Example 4.6. Moreover, we prove the existence and nonexistence of warping function $f$ for Einstein Poisson warped product space under by taking the various dimensions of base space $B$ (i.e; (1). $dim B=1,$ (2). $dim B\geq3$).
\section{\textbf{Preliminaries}}
 Let $(M,\Pi)$ is a Poisson manifold where $\Pi$ be the bivector field, an anchor map(sharp map) $\sharp_{\Pi}:T^*M\rightarrow TM$  associated $\Pi$ defined by
$$\beta(\sharp_{\Pi}(\alpha))=\Pi(\alpha,\beta),\qquad\quad\quad\quad\text{for any}\:\alpha,\beta\in T^*M.$$
The Koszul bracket on differential 1-form $\Gamma(T^*M)$ is a bracket $[.,.]_\Pi$ defined by
$$[\alpha,\beta]_\Pi=\mathcal{L}_{\sharp_{\Pi}(\alpha)}\beta-\mathcal{L}_{\sharp_{\Pi}(\beta)}\alpha-d(\Pi(\alpha,\beta)),\qquad\quad\text{for any}\:\alpha,\beta\in\Gamma(T^*M).$$
Moreover, if $\Pi$ is a Poisson tensor (see \cite{ivs}), then Koszul bracket form a Lie bracket and sharp map $\sharp_{\Pi}$ provides a Lie algebra homomorphism $\sharp_{\Pi}:\Gamma(T^*M)\rightarrow \Gamma(TM)$ i.e,
$$\sharp_{\Pi}([\alpha,\beta]_\Pi)=[\sharp_{\Pi}(\alpha),\sharp_{\Pi}(\beta)], \qquad\quad\quad\quad\text{for any}\:\alpha,\beta\in\Gamma(T^*M),$$
whrere $[.,.]$ is the usual Lie bracket on $\Gamma(TM)$.\\

Let $(M,\Pi)$ be a Poisson manifold, the contravariant connection $\mathcal{D}$ associated to $\Pi$ is an $\mathbb{R}$-bilinear map $\mathcal{D}:\Gamma(T^*M)\times\Gamma(T^*M)\rightarrow \Gamma(T^*M)$ given by $(\alpha,\beta)\mapsto\mathcal{D}_\alpha\beta$ such that $$\mathcal{D}_{f\alpha}\beta=f\mathcal{D}_\alpha\beta\quad\text{and}\quad\mathcal{D}_\alpha(f\beta)=f\mathcal{D}_\alpha\beta+\sharp_{\Pi}(\alpha)(f)\beta,$$
for any $\alpha,\beta\in\Gamma(T^*M)$ and for any $f\in\mathcal{C}^\infty(M).$ The curvature and torsion tensor with respect to contravariant derivative $\mathcal{D}$ are formally identical to the covariant case i.e, for any $\alpha,\beta,\gamma\in\Gamma(T^*M)$,
\begin{align}
\mathcal{T}(\alpha,\beta)&=\mathcal{D}_\alpha\beta-\mathcal{D}_\beta\gamma-[\alpha,\beta]_\Pi,\label{e250}\\
\mathcal{R}(\alpha,\beta)\gamma&=\mathcal{D}_\alpha\mathcal{D}_\beta\gamma-\mathcal{D}_\beta\mathcal{D}_\alpha\gamma-\mathcal{D}_{[\alpha,\beta]_\Pi}\gamma,\label{e251}
\end{align}
where $\mathcal{T}$ and $\mathcal{R}$ are $(2,1)$ and $(3,1)$-type tensor respectively and $\mathcal{D}$ is called torsion-free(respectively, flat) if $\mathcal{T}=0$ (respectively, $\mathcal{R}=0$).\\

Let $(M,\tilde{g})$ be a pseudo-Riemannian manifold, the bundle isomorphism associate to $\tilde{g}$ is $b_{\tilde{g}}:TM\rightarrow T^*M$ given by $X\mapsto\tilde{g}(X,.)$ and its inverse $\sharp_{g}$ provides a metric $g$ on cotangent bundle  by $g(\alpha,\beta)=\tilde{g}(\sharp_g(\alpha),\sharp_g(\beta)).$
Here the metric $g$ is called cometric of the metric $\tilde{g}.$ If $\Pi$ is a bivector field on pseudo-Riemannian manifold $(M,\tilde{g})$ and $J:T^*M\rightarrow T^*M$ is the field endomorphism then
$$\Pi(\alpha,\beta)=g(J\alpha,\beta)=-g(\alpha,J\beta).$$\par
Let $(M,\Pi)$ be a Poisson manifold associated with pseudo-Riemannian metric $g$ on $T^*M$, existence and uniqueness of a contravariant connection $\mathcal{D}$ according to \cite{mbo} by
\begin{align}
2g(\mathcal{D}_\alpha\beta,\gamma)=&\sharp_{\Pi}(\alpha)g(\beta,\gamma)+\sharp_{\Pi}(\beta)g(\alpha,\gamma)-\sharp_{\Pi}(\gamma)g(\alpha,\beta)\nonumber\\
+&g([\alpha,\beta]_\Pi,\gamma)-g([\beta,\gamma]_\Pi,\alpha)+g([\gamma,\alpha]_\Pi,\beta),\label{e25}
\end{align}
satisfies following two properties,
\begin{align*}
&\textbf{(i)}.\:\mathcal{D}_\alpha\beta-\mathcal{D}_\beta\alpha=[\alpha,\beta]_\Pi;(\text{torsion-free}),\\
&\textbf{(ii)}.\:\sharp_{\Pi}(\alpha).g(\beta,\gamma)=g(\mathcal{D}_\alpha\beta,\gamma)+g(\beta,\mathcal{D}_\alpha\gamma);(\text{metric condition}),
\end{align*}
 for any $\alpha,\beta,\gamma\in\Gamma(T^*M)$. With the notations above $\mathcal{D}$ is called contravariant Levi-Civita connection associated to pair $(\Pi,g)$ on Poission manifold $(M,\Pi)$. \\
Let $(M,\Pi)$ be the Poisson manifold and $\mathcal{D}$ is contravariant Levi-Civita connection associated to $(\Pi,g)$, for given a point $p\in M$ the Ricci curvature $Ric_p$ and scalar curvature $S_p$ of $M$ at point $p$ given by
\begin{align}
Ric_p(\alpha,\beta)&=\displaystyle\sum_{i=1}^{d}g_p(\mathcal{R}_p(\alpha,\eta_i)\eta_i,\beta),\quad \text{for any}\:\alpha,\beta\in T_p^*M,\label{e27}\\
S_p&=\displaystyle\sum_{i=1}^{d}Ric_p(\eta_i,\eta_i),\label{e270}
\end{align}
where $\{\eta_i\}$ is any local orthonormal basis  of $T_p^*M$ and $Ric$ is a $(2,0)$-type symmetric tensor.\\

Let $\mathcal{D}$ is contravariant Levi-Civita connection associated to $(\Pi,g)$ on Poisson manifold $(M,\Pi)$, then the triplet $(M,g,\Pi)$ is said to be Riemannian Poisson manifold if $\mathcal{D}\Pi=0$ i,e. for any $\alpha,\beta,\gamma\in\Gamma(T^*M)$,
\begin{equation}
\sharp_{\Pi}(\alpha).\Pi(\beta,\gamma)-\Pi(\mathcal{D}_\alpha\beta,\gamma)-\Pi(\beta,\mathcal{D}_\alpha\gamma)=0.
\end{equation}
If $J$ is a field endomorphism on Riemannian Poisson manifold $(M,\Pi,g)$ then $\mathcal{D}J=0$ i.e,
$$\mathcal{D}_\alpha(J\beta)=J\mathcal{D}_\alpha\beta,\quad\quad \text{for any}\:\alpha,\beta\in T^*M.$$\par
Let $\mathcal{D}$ is contravariant Levi-Civita connection associated to $(\Pi,g)$ on Poisson manifold $(M,\Pi)$, the contravariant Laplacian operator associated to $\mathcal{D}$ over any tensor field $T$ on $M$ defined in [\cite{zs}, p. 9] by
\begin{equation}
\Delta^\mathcal{D}(T)=-\displaystyle\sum_{i=1}^{d}\mathcal{D}^2_{\eta_i,\eta_i}T,
\end{equation}
where $\{\eta_i\}$ is any local orthonormal coframe field on $M$. If $f\in C^{\infty}(M)$ then from [\cite{yar}, Proposition 1] and equation $(2.5),$ provides
\begin{equation}
\Delta^\mathcal{D}(f)=-\displaystyle\sum_{i=1}^{d}H_{\Pi}^f(\eta_i,\eta_i)=\displaystyle\sum_{i=1}^{d}g(\mathcal{D}_{\eta_i}Jdf,\eta_i),
\end{equation}
where $\{\eta_i\}$ is any local orthonormal coframe field on $M$.
\subsection{Bivector and cometric on warped product space}
From (\cite{yar},p. 287), Let $\Pi_B$ and $\Pi_F$ are the bivectors field on  Riemannian manifolds $(B,\tilde{g}_B)$ and $(F,\tilde{g}_F)$ with cometrics $g_B$ and $g_F$ respectively. The bivector field on the product space $B\times F$ is a unique bivector field $\Pi=\Pi_B+\Pi_F$ such that
$$\Pi(\alpha_1^h,\beta_1^h)=\Pi_B(\alpha_1,\beta_1)^h,\quad \Pi(\alpha_1^h,\beta_2^v)=0,\quad \Pi(\alpha_2^v,\beta_2^v)=\Pi_F(\alpha_2,\beta_2)^v,$$
for any $\alpha_1,\beta_1\in\Gamma(T^*B)$ and $\alpha_2,\beta_2\in\Gamma(T^*F)$.
\begin{proposition}
Let $\alpha_1,\beta_1\in\Gamma(T^*B)$ and $\alpha_2,\beta_2\in\Gamma(T^*F)$. Let $\alpha=\alpha_1^h+\alpha_2^v$ and $\beta=\beta_1^h+\beta_2^v$, we have
    \begin{align*}
	&\textbf{(a).}\:\sharp_{\Pi}(\alpha)=[\sharp_{\Pi_B}(\alpha_1)]^h+[\sharp_{\Pi_F}(\alpha_2)]^v,\\
	&\textbf{(b).}\:\mathcal{L}_{{\Pi}(\alpha)}\beta=[\mathcal{L}_{{\Pi_B}(\alpha_1)}\beta_1]^h+[\mathcal{L}_{{\Pi_F}(\alpha_2)}\beta_2]^v,\\
	&\textbf{(c).}\:[\alpha,\beta]=[\alpha_1,\beta_1]_{\Pi_B}^h+[\alpha_2,\beta_2]_{\Pi_F}^v.
	\end{align*}
\end{proposition}
From (\cite{ddo},p. 23), we recall the cometric $g^f$ of the warped metric $$\tilde{g}^f=\tilde{g_{B}}+f^2\tilde{g_{F}},$$
on product space $B\times F$
where $f$ be a positive smooth function on $B$.
The warped metric $\tilde{g}^f$ explicitly can be written as,
for any $X_1,Y_1\in\Gamma(TB)$ and $X_2,Y_2\in\Gamma(TF)$,
\begin{eqnarray}
\left\{
\begin{array}{ll}
\tilde{g}^f(X_1^h,Y_1^h)=\tilde{g}_B(X_1,Y_1)^h,\\
\tilde{g}^f(X_1^h,Y_2^v)=\tilde{g}_F(X_2^v,Y_1^h)=0,\\
\tilde{g}^f(X_2^v,Y_2^v)=(f^h)^2\tilde{g}_F(X_2,Y_2)^v.
\end{array}
\right.
\end{eqnarray}
With the above notation, we have the cometric $g$ of the metric $\tilde{g}$ is
$$g^f=g_B^h+\frac{1}{(f^h)^2}g_F^v.$$
Equivalently, for any $\alpha_1,\beta_1\in\Gamma(T^*B)$ and $\alpha_2,\beta_2\in\Gamma(T^*F)$,
we have
\begin{eqnarray}
\left\{
\begin{array}{ll}
g^f(\alpha_1^h,\beta_1^h)=g_B(\alpha_1,\beta_1)^h,\\
g^f(\alpha_1^h,\beta_2^v)=g_B(\alpha_2^v,\beta_1^h)=0,\\
g^f(\alpha_2^v,\beta_2^v)=\frac{1}{(f^h)^2}g_F(\alpha_2,\beta_2)^v.
\end{array}
\right.
\end{eqnarray}
\begin{definition}
The ordered pair $(M=B\times_f F, g^f)$ called contravariant warped product space of the warped product space $(M=B\times_f F, \tilde{g}^f)$.
\end{definition}
Now, we will state theorem which informs that under what condition contravariant warped product of two pseudo-Riemannian Poisson manifold become pseudo-Riemannian Poisson manifold. If $\Pi=\Pi_B+\Pi_F$ then from Theorem 5.2 of (\cite{yar}, p. 249), we have
\begin{theorem}
	If $f$ is a Casimir function, then the triple $(B\times_f F,g^f,\Pi)$ is a pseudo-Riemannian Poisson warped product space if and only if $(B,g_B,\Pi_B)$ and $(F,g_F,\Pi_F)$ both are pseudo-Riemannian Poisson manifold.
\end{theorem}
\section{\textbf{The contravariant Levi-Civita connection and curvatures on contravariant warped product space with product bivector field}}
Throughout we will denote $\mathcal{D}^B$ and $\mathcal{D}^F$ for contravariant Levi-Civita connection on the Poisson manifolds $(B,\Pi_B)$ and $(F,\Pi_F)$ respectively. Accordingly, we will notify $\mathcal{R}_B$, $Ric_B$, $S_B$ etc. for corresponding curvatures, Ricci curvatures, scalar curvatures etc.\par
In the following two propositions, we have to calculate the contravariant Levi-Civita connection $\mathcal{D}$ associated with the pair $(g^f,\Pi)$ $(\text{where}\  g^f=g_B^h+\frac{1}{(f^h)^2}g_F^v\ \text{and}\ \Pi=\Pi_1+\Pi_2)$ and corresponding curvature $\mathcal{R}$ on contravariant warped product space $(M=B\times_f F, g^f)$ respectively.
\begin{proposition}
	Let $(M=B\times_f F, g^f)$ be a contravariant warped product space then
	for any $\alpha_1,\beta_1\in\Gamma(T^*B)$ and $\alpha_2,\beta_2\in\Gamma(T^*F)$, we have
	\begin{align*}
	\textbf{(a).}\:\mathcal{D}_{\alpha_1^h}\beta_1^h&=(\mathcal{D}_{\alpha_1}^B\beta_1)^h,\\
	\textbf{(b).}\:\mathcal{D}_{\alpha_2^v}\beta_2^v&=(\mathcal{D}_{\alpha_2}^F\beta_2)^v-\frac{1}{(f^h)^3}g_F(\alpha_2,\beta_2)^v(J_1df)^h,\\
	\textbf{(c).}\:\mathcal{D}_{\alpha_1^h}\beta_2^v&=\mathcal{D}_{\beta_2^v}\alpha_1^h=\frac{1}{f^h}g_B(J_1df,\alpha_1)^h\beta_2^v.
	\end{align*}
\end{proposition}
\begin{proof}
	Let $\alpha_1,\beta_1,\gamma_1\in\Gamma(T^*B)$, $\alpha_2,\beta_2,\gamma_2\in\Gamma(T^*F)$ and $\alpha=\alpha_1^h+\alpha_2^v$,  $\beta=\beta_1^h+\beta_2^v$, $\gamma=\gamma_1^h+\gamma_2^v$ then from equation $(2.3),$ we have
	\begin{align}
	2g^f(\mathcal{D}_\alpha\beta,\gamma)&=\sharp_{\Pi}(\alpha)g^f(\beta,\gamma)+\sharp_{\Pi}(\beta)g^f(\alpha,\gamma)-\sharp_{\Pi}(\gamma)g^f(\alpha,\beta)\nonumber\\
	&+g^f([\alpha,\beta]_{\Pi},\gamma)-g^f([\beta,\gamma]_{\Pi},\alpha)+g^f([\gamma,\alpha]_{\Pi},\beta).
	\end{align}
	$(a).\:$ Putting $\alpha=\alpha_1^h$, $\beta=\beta_1^h$ and $\gamma=\gamma_1^h$ in $(3.1)$, after that using the system of equations $(2.10)$ and Proposition 2.1, we obtain
	\begin{equation*}
	2g^f(\mathcal{D}_{\alpha_1^h}\beta_1^h,\gamma_1^h)=2(g_B(\mathcal{D}_{\alpha_1}^B\beta_1,\gamma_1))^h.
	\end{equation*}
	Again applying the first equation of $(2.10)$ in this equation, provides
	\begin{equation*}
	g^f(\mathcal{D}_{\alpha_1^h}\beta_1^h,\gamma_1^h)=g^f((\mathcal{D}_{\alpha_1}^B\beta_1)^h,\gamma_1^h).
	\end{equation*}
	Similarly, taking $\alpha=\alpha_1^h$, $\beta=\beta_1^h$ and $\gamma=\gamma_2^v,$ we get $g^f(\mathcal{D}_{\alpha_1^h}\beta_1^h,\gamma_2^v)=0$ and then
	\begin{equation*}
	g^f(\mathcal{D}_{\alpha_1^h}\beta_1^h,\gamma_1^v)=g^f((\mathcal{D}_{\alpha_1}^B\beta_1)^h,\gamma_2^v).
	\end{equation*}
	The result follows.\\\\
	$(b).\:$ Putting $\alpha=\alpha_2^v$, $\beta=\beta_2^v$ and $\gamma=\gamma_1^h$ in $(3.1)$, after that using the system of equations $(2.10)$ and Proposition 2.1, we have
	\begin{align*}
	g^f(\mathcal{D}_{\alpha_2^v}\beta_2^v,\gamma_1^h)&=\frac{1}{2}\{g^f([\alpha_2^v,\beta_2^v]_{\Pi},\gamma_1^h)-\sharp_{\Pi}(\gamma_1^h)g^f(\alpha_2^v,\beta_2^v)\}\\
	&=-\frac{1}{(f^h)^3}g_F(\alpha_2,\beta_2)^vg_B(J_1df,\gamma_1)^h\\
	&=g^f(-\frac{1}{(f^h)^3}g_F(\alpha_2,\beta_2)^v(J_1df)^h,\gamma_1^h).
	\end{align*} 
	Similarly, taking $\alpha=\alpha_2^v$, $\beta=\beta_2^v$ and $\gamma=\gamma_2^v$, we get
	\begin{equation*}
	g^f(\mathcal{D}_{\alpha_2^v}\beta_2^v,\gamma_2^v)=g^f((\mathcal{D}_{\alpha_2}\beta_2)^v,\gamma_2^v).
	\end{equation*}
	The result follows.\\\\
	$(c).\:$ Since $\mathcal{D}$ is torsion-free therefore $\mathcal{D}_{\alpha_1^h}\beta_2^v-\mathcal{D}_{\beta_2^v}\alpha_1^h=[\alpha_1^h,\beta_2^v]_{\Pi}$. By Proposition 2.2 we have $[\alpha_1^h,\beta_2^v]_{\Pi}=0$, thus $\mathcal{D}_{\alpha_1^h}\beta_2^v=\mathcal{D}_{\beta_2^v}\alpha_1^h$. The remaining proof is similar to the proofs of $(a)$ and $(b)$.
\end{proof}
\begin{proposition}Let $(M=B\times_f F, g^f)$ be a contravariant warped product space then for any $\alpha_1,\beta_1,\gamma_1\in\Gamma(T^*B)$, $\alpha_2,\beta_2,\gamma_2\in\Gamma(T^*F)$ and $\gamma=\gamma_1^{h}+\gamma_2^{v}$, we have
\begin{align*}
\textbf{(a).}\:\mathcal{R}(\alpha_1^{h},\beta_1^{h})\gamma&=\big[\mathcal{R}_B(\alpha_1,\beta_1)\gamma_1\big]^h\\
&+\frac{1}{f^h}\big[g_B(\mathcal{D}_{\alpha_1}^BJ_1df,\beta_1)-g_B(\mathcal{D}_{\beta_1}^BJ_1df,\alpha_1)\big]^h\gamma_2^v\\
&+\frac{1}{(f^h)^2}\big[\mathcal{D}_{\beta_1}^B(f)g_B(J_1df,\alpha_1)-\mathcal{D}_{\alpha_1}^B(f)g_B(J_1df,\beta_1)\big]^h\gamma_2^v,\\
\textbf{(b).}\:\mathcal{R}(\alpha_1^{h},\beta_2^{v})\gamma_1^{h}&=\frac{1}{(f^h)^{2}}\big[g_B(J_1df,\alpha_1)g_B(J_1df,\gamma_1)\big]^h\beta_2^{v}+g_B(\mathcal{D}_{\alpha_1}^B(\frac{J_1df}{f}),\gamma_1)^h\beta_2^{v},\\
\textbf{(c).}\:\mathcal{R}(\alpha_1^{h},\beta_2^{v})\gamma_2^v&=-g_F(\beta_2,\gamma_2)^v\Big(\frac{1}{f^3}(\mathcal{D}_{\alpha_1}^BJ_1df)+\frac{2}{f^4}g_B(J_1df,\alpha_1)J_1df\Big)^h,\\
\textbf{(d).}\:\mathcal{R}(\alpha_2^{v},\beta_2^{v})\gamma_1^h&=0,\\
\textbf{(e).}\:\mathcal{R}(\alpha_2^{v},\beta_2^{v})\gamma_2^v&=\big[\mathcal{R}_B(\alpha_2,\beta_2)\gamma_2\big]^h+\Big(\frac{||J_1df||_B^{2}}{f^4}\Big)^h\big[g_F(\alpha_2,\gamma_2)\beta_2-g_F(\beta_2,\gamma_2)\alpha_2\big]^v.
\end{align*}
\end{proposition}
\begin{proof}
	Let $\alpha_1,\beta_1,\gamma_1\in\Gamma(T^*B)$, $\alpha_2,\beta_2,\gamma_2\in\Gamma(T^*F)$ and $\alpha=\alpha_1^h+\alpha_2^v$,  $\beta=\beta_1^h+\beta_2^v$, $\gamma=\gamma_1^h+\gamma_2^v$ then
	\begin{equation}
	\mathcal{R}(\alpha,\beta)\gamma=\mathcal{D}_\alpha\mathcal{D}_\beta\gamma-\mathcal{D}_\beta\mathcal{D}_\alpha\gamma-\mathcal{D}_{[\alpha,\beta]_{\Pi}}\gamma.
	\end{equation}
	$(a).\:$ Putting $\alpha=\alpha_1^h$, $\beta=\beta_1^h$ and $\gamma=\gamma_1^h$ in $(3.2),$ we have
	\begin{equation*}
	\mathcal{R}(\alpha_1^h,\beta_1^h)\gamma_1^h=\mathcal{D}_{\alpha_1^h}(\mathcal{D}_{\beta_1^h}\gamma_1^h)-\mathcal{D}_{\beta_1^h}(\mathcal{D}_{\alpha_1^h}\gamma_1^h)-\mathcal{D}_{[\alpha_1^h,\beta_1^h]_{\Pi}}\gamma_1^h.
	\end{equation*}
	Using Proposition 3.1 and Proposition 2.1, in this equation provides
	\begin{align*}
	\mathcal{R}(\alpha_1^h,\beta_1^h)\gamma_1^h&=\mathcal{D}_{\alpha_1^h}(\mathcal{D}_{\beta_1}^B\gamma_1)^h-\mathcal{D}_{\beta_1^h}(\mathcal{D}_{\alpha_1}^B\gamma_1)^h-\mathcal{D}_{[\alpha_1,\beta_1]_{\Pi_B}^h}\gamma_1^h\nonumber\\
	&=(\mathcal{D}_{\alpha_1}^B\mathcal{D}_{\beta_1}^B\gamma_1)^h-(\mathcal{D}_{\beta_1}^B\mathcal{D}_{\alpha_1}^B\gamma_1)^h-(\mathcal{D}_{[\alpha_1,\beta_1]_{\Pi_B}}^1\gamma_1)^h\nonumber\\
	&=\big[\mathcal{R}_B(\alpha_1,\beta_1)\gamma_1\big]^h.
	\end{align*}
	Putting $\alpha=\alpha_1^h$, $\beta=\beta_1^h$ and $\gamma=\gamma_2^v$ in $(3.2)$ and using Proposition 2.1, we have
	\begin{equation}
	\mathcal{R}(\alpha_1^h,\beta_1^h)\gamma_2^v=\mathcal{D}_{\alpha_1^h}(\mathcal{D}_{\beta_1^h}\gamma_2^v)-\mathcal{D}_{\beta_1^h}(\mathcal{D}_{\alpha_1^h}\gamma_2^v)-\mathcal{D}_{[\alpha_1,\beta_1]_{\Pi_{B}}^h}\gamma_2^v.
	\end{equation}
	Applying Proposition 3.1 in the first term $T_1$ of $(3.3),$ we have
	\begin{align}
	T_1&=\mathcal{D}_{\alpha_1^h}(\mathcal{D}_{\beta_1^h}\gamma_2^v)\nonumber\\
	&=\mathcal{D}_{\alpha_1^h}\Big(\frac{1}{f^h}g_B(J_1df,\beta_1)^h\gamma_2^v\Big)\nonumber\\
	&=\frac{1}{(f^h)^2}\big[\{-\mathcal{D}_{\alpha_1}^B(f)+g_B(J_1df,\alpha_1)\}g_B(J_1df,\beta_1)+f\mathcal{D}_{\alpha_1}^B\big(g_B(J_1df,\beta_1)\big)\big]^h\gamma_2^v.
	\end{align}
	Interchanging $\alpha_1$ and $\beta_1$ in this equation provides the second term $T_2$ of $(3.3)$ and so
	\begin{align*}
	T_1-T_2&=\mathcal{D}_{\alpha_1^h}(\mathcal{D}_{\beta_1^h}\gamma_2^v)-\mathcal{D}_{\beta_1^h}(\mathcal{D}_{\alpha_1^h}\gamma_2^v)\\
	&=\Big(\frac{\mathcal{D}_{\beta_1}^B(f)g_B(J_1df,\alpha_1)}{f^2}-\frac{\mathcal{D}_{\alpha_1}^B(f)g_B(J_1df,\beta_1)}{f^2}+\frac{\mathcal{D}_{\alpha_1}^B\big(g_B(J_1df,\beta_1)\big)}{f}\\
	&-\frac{\mathcal{D}_{\beta_1}^B\big(g_B(J_1df,\alpha_1)\big)}{f}\Big)^h\gamma_2^v.
	\end{align*}
	Applying Proposition 3.1 in the third term $T_3$ of $(3.3),$ we have
	\begin{align*}
	T_3&=\mathcal{D}_{[\alpha_1,\beta_1]_{\Pi_{B}}^h}\gamma_2^v\\
	&=\Big(\frac{g_B(J_1df,\mathcal{D}_{\alpha_1}^B\beta_1)}{f}-\frac{g_B(J_1df,\mathcal{D}_{\beta_1}^B\alpha_1)}{f}\Big)^h\gamma_2^v.
	\end{align*}
	Since,
	\begin{equation*}
	\mathcal{R}(\alpha_1^h,\beta_1^h)\gamma=\mathcal{R}(\alpha_1^h,\beta_1^h)\gamma_1^h+\mathcal{R}(\alpha_1^h,\beta_1^h)\gamma_2^v.
	\end{equation*}
	Thus after some calculations, the result follows.\\
	$(b).\:$ Putting $\alpha=\alpha_1^h$, $\beta=\beta_2^v$ and $\gamma=\gamma_1^h$ in $(3.2)$, we have
	\begin{equation*}
	\mathcal{R}(\alpha_1^h,\beta_2^v)\gamma_1^h=\mathcal{D}_{\alpha_1^h}(\mathcal{D}_{\beta_2^v}\gamma_1^h)-\mathcal{D}_{\beta_2^v}(\mathcal{D}_{\alpha_1^h}\gamma_1^h)-\mathcal{D}_{[\alpha_1^h,\beta_2^v]_{\Pi}}\gamma_1^h.
	\end{equation*}
	Using Proposition 3.1 and Proposition 2.1 in this equation, we have
	\begin{equation}
	\mathcal{R}(\alpha_1^h,\beta_2^v)\gamma_1^h=\mathcal{D}_{\alpha_1^h}(\mathcal{D}_{\gamma_1^h}\beta_2^v)-\mathcal{D}_{\beta_2^v}(\mathcal{D}_{\alpha_1}^B\gamma_1)^h.
	\end{equation}
	Replacing $\beta_1=\gamma_1$ and $\gamma_2=\beta_2$ in $(3.4)$, we have the first term $\mathcal{D}_{\alpha_1^h}(\mathcal{D}_{\gamma_1^h}\beta_2^v)$ of $(3.5)$. The second term of $(3.5)$ is given by
	\begin{align*}
	\mathcal{D}_{\beta_2^v}(\mathcal{D}_{\alpha_1}^B\gamma_1)^h=\frac{1}{2f^h}\big[2g_B(J_1df,\mathcal{D}_{\alpha_1}^B\gamma_1)^h\beta_2^v-\{f^3g_B(d\mu,\mathcal{D}_{\alpha_1}^B\gamma_1)\}^h(J_2\beta_2)^v\big].
	\end{align*}
	Thus after some calculations, the result follows.\\
	$(c).\:$ Putting $\alpha=\alpha_1^h$, $\beta=\beta_2^v$ and $\gamma=\gamma_2^v$ in $(3.2),$ we have
	\begin{equation*}
	\mathcal{R}(\alpha_1^h,\beta_2^v)\gamma_2^v=\mathcal{D}_{\alpha_1^h}(\mathcal{D}_{\beta_2^v}\gamma_2^v)-\mathcal{D}_{\beta_2^v}(\mathcal{D}_{\alpha_1^h}\gamma_2^v)-\mathcal{D}_{[\alpha_1^h,\beta_2^v]_{\Pi}}\gamma_2^v.
	\end{equation*}
	Using Proposition 2.1 in this equation, we have
	\begin{equation}
	\mathcal{R}(\alpha_1^h,\beta_2^v)\gamma_2^v=\mathcal{D}_{\alpha_1^h}(\mathcal{D}_{\beta_2^v}\gamma_2^v)-\mathcal{D}_{\beta_2^v}(\mathcal{D}_{\alpha_1^h}\gamma_2^v).
	\end{equation}
	Applying Proposition 3.1 in the first term of $(3.6),$ provides
	\begin{align*}
	\mathcal{D}_{\alpha_1^h}(\mathcal{D}_{\beta_2^v}\gamma_2^v)&=\mathcal{D}_{\alpha_1^h}\Big((\mathcal{D}_{\beta_2}^F\gamma_2)^v-\frac{1}{(f^h)^3}g_F(\beta_2,\gamma_2)^v(J_1df)^h\Big)\\
	&=\Big(\frac{g_B(J_1df,\alpha_1)}{f}\Big)^h(\mathcal{D}_{\beta_2}^F\gamma_2)^v\\
	&-\frac{1}{(f^h)^3}g_F(\beta_2,\gamma_2)^v(\mathcal{D}_{\alpha_1}^BJ_1df)^h-3\Big(\frac{g_B(J_1df,\alpha_1)}{f^4}\Big)^hg_F(\beta_2,\gamma_2)^v(J_1df)^h.
	\end{align*}
	The second term of $(3.6)$ is given by 
	\begin{align}
	\mathcal{D}_{\beta_2^v}(\mathcal{D}_{\alpha_1^h}\gamma_2^v)&=\mathcal{D}_{\beta_2^v}\Big(\frac{1}{f^h}g_B(J_1df,\alpha_1)^h\gamma_2^v\Big)\nonumber\\
	&=\Big(\frac{g_B(J_1df,\alpha_1)}{f}\Big)^h\Big((\mathcal{D}_{\beta_2}^F\gamma_2)^v-\frac{1}{(f^h)^3}g_F(\beta_2,\gamma_2)^v(J_1df)^h\Big).
	\end{align}
	Using the above terms in $(3.6)$ and after some calculations, the result follows.\\
	$(d).\:$ Putting $\alpha=\alpha_2^v$, $\beta=\beta_2^v$ and $\gamma=\gamma_1^h$ in $(3.2),$ we have
	\begin{equation*}
	\mathcal{R}(\alpha_2^v,\beta_2^v)\gamma_1^h=\mathcal{D}_{\alpha_2^v}(\mathcal{D}_{\beta_2^v}\gamma_1^h)-\mathcal{D}_{\beta_2^v}(\mathcal{D}_{\alpha_2^v}\gamma_1^h)-\mathcal{D}_{[\alpha_2^v,\beta_2^v]_{\Pi}}\gamma_1^h.
	\end{equation*}
	Using Proposition 3.1 in this equation, provides
	\begin{equation}
	\mathcal{R}(\alpha_2^v,\beta_2^v)\gamma_1^h=\mathcal{D}_{\alpha_2^v}(\mathcal{D}_{\gamma_1^h}\beta_2^v)-\mathcal{D}_{\beta_2^v}(\mathcal{D}_{\gamma_1^h}\alpha_2^v)-\mathcal{D}_{[\alpha_2^v,\beta_2^v]_{\Pi}}\gamma_1^h.
	\end{equation}
	The first and second term of $(3.8)$ are analogous of $(3.7)$.\\
	Using Proposition 3.1 and Proposition 2. in the third term of $(3.8),$ we have
	\begin{align*}
	\mathcal{D}_{[\alpha_2^v,\beta_2^v]_{\Pi}}\gamma_1^h&=\mathcal{D}_{[\alpha_2,\beta_2]^v_{\Pi_F}}\gamma_1^h\\
	&=\mathcal{D}_{\gamma_1^h}[\alpha_2,\beta_2]^v_{\Pi_F}\\
	&=\Big(\frac{g_B(J_1df,\gamma_1)}{f}\Big)^h[\alpha_2,\beta_2]^v_{\Pi_F}.
	\end{align*}
	After some calculations, the result follows.\\
	$(e).\:$ Putting $\alpha=\alpha_2^v$, $\beta=\beta_2^v$ and $\gamma=\gamma_2^v$ in $(3.2),$ we have
	\begin{equation}
	\mathcal{R}(\alpha_2^v,\beta_2^v)\gamma_2^v=\mathcal{D}_{\alpha_2^v}(\mathcal{D}_{\beta_2^v}\gamma_2^v)-\mathcal{D}_{\beta_2^v}(\mathcal{D}_{\alpha_2^v}\gamma_2^v)-\mathcal{D}_{[\alpha_2^v,\beta_2^v]_{\Pi}}\gamma_2^v.
	\end{equation}
	Applying Proposition 3.1 in the first term $P_1$ of $(3.9),$ provides
	\begin{align*}
	P_1&=\mathcal{D}_{\alpha_2^v}(\mathcal{D}_{\beta_2^v}\gamma_2^v)\\
	&=\mathcal{D}_{\alpha_2^v}\Big((\mathcal{D}_{\beta_2}^F\gamma_2)^v-\frac{1}{(f^h)^3}g_F(\beta_2,\gamma_2)^v(J_1df)^h\Big)\\
	&=(\mathcal{D}^F_{\alpha_2}\mathcal{D}^F_{\beta_2}\gamma_2)^v-\frac{1}{(f^h)^3}\big[g_F(\alpha_2,\mathcal{D}^F_{\beta_2}\gamma_2)+\mathcal{D}^F_{\alpha_2}\big(g_F(\beta_2,\gamma_2)\big)\big]^v(J_1df)^h\\
	&-\Big(\frac{||J_1df||_B^{2}}{f^4}\Big)^hg_F(\beta_2,\gamma_2)^v\alpha_2^v.
	\end{align*}
	Interchanging $\alpha_2$ and $\beta_2$ in this equation provides the second term $P_2=\mathcal{D}_{\beta_2^v}(\mathcal{D}_{\alpha_2^v}\gamma_2^v)$ of $(3.9)$. The third term $P_3$ of $(3.9)$ is given by
	\begin{align*}
	P_3&=\mathcal{D}_{[\alpha_2^v,\beta_2^v]_{\Pi}}\gamma_2^v\\
	&=\mathcal{D}_{[\alpha_2,\beta_2]^v_{\Pi_F}}\gamma_2^v\\
	&=(\mathcal{D}^F_{[\alpha_2,\beta_2]_{\Pi_F}}\gamma_2)^v-\frac{1}{(f^h)^3}g_F([\alpha_2,\beta_2]_{\Pi_F},\gamma_2)^v(J_1df)^h.
	\end{align*}
	Using the above terms in $(3.9)$, after some calculations the result follows.
\end{proof}
In the next two propositions, we will provide the expression for the Ricci curvature $Ric$ and scalar curvature $S$ with respect to the local $g^f$-orthonormal basis $$\{dx_1^{h},...,dx_{s_1}^{h},f^hdy_1^{v},...,f^hdy_{s_2}^{v}\}$$  on open subset $U_1\times U_2$ of $B\times F$, whenever $\{dx_1,...,dx_{s_1}\}$ and $\{dy_1,...,dy_{s_2}\}$ are local $g_B$-orthonormal basis  on an open subset $U_1$ of $B$ and local $g_F$-orthonormal basis  on an open subset $U_2$ of $F$ respectively.
\begin{proposition}Let $(M=B\times_f F, g^f)$ be a contravariant warped product space then
	for any $\alpha_1,\beta_1\in\Gamma(T^*B)$ and $\alpha_2,\beta_2\in\Gamma(T^*F)$, we have
	\begin{align*}
	 \textbf{(a).}\:Ric(\alpha_1^{h},\beta_1^{h})&=Ric_B(\alpha_1,\beta_1)^h-\frac{s_2}{(f^h)^2}\big[2A(\alpha_1,\beta_1)+fg_B(\mathcal{D}_{\alpha_1}^B(J_1df),\beta_1)\big]^h\\
	&\quad\quad where \: A(.,.)=g_B(J_1df,.)g_B(J_1df,.),\\
	\textbf{(b).}\:Ric(\alpha_1^{h},\beta_2^{v})&=0,\\
	 \textbf{(c).}\:Ric(\alpha_2^{v},\beta_2^{v})&=Ric_F(\alpha_2,\beta_2)^v-\Big(\frac{(s_2+1)||J_1df||_B^{2}}{f^4}+\frac{\Delta^{\mathcal{D}^B}(f)}{f^3}\Big)^hg_F(\alpha_2,\beta_2)^v,
	\end{align*}
	$\text{where}\quad\Delta^{\mathcal{D}^B}(f)=\displaystyle\sum_{i=1}^{s_1}g_B(\mathcal{D}_{dx_i}^B(J_1df),dx_i)$.
    \end{proposition}
\begin{proof}
	As $\{dx_1^{h},...,dx_{s_1}^{h},f^hdy_1^{v},...,f^hdy_{s_2}^{v}\}$ is a local $g^f$-orthonormal basis then
	\begin{align*}
	Ric(\alpha_1^h,\beta_1^h)&=\sum_{i=1}^{s_1}g^f(\mathcal{R}(\alpha_1^h,dx_{i}^h)dx_{i}^h,\beta_1^h)+\sum_{j=1}^{s_2}g^f(\mathcal{R}(\alpha_1^h,f^hdy_{j}^v)f^hdy_{j}^v,\beta_1^h)\\
	&=\sum_{i=1}^{s_1}g^f(\mathcal{R}(\alpha_1^h,dx_{i}^h)dx_{i}^h,\beta_1^h)+(f^h)^2\sum_{j=1}^{s_2}g^f(\mathcal{R}(\alpha_1^h,dy_{j}^v)dy_{j}^v,\beta_1^h).\\
	\end{align*}
	Using Proposition 3.2 and system of equations (2.10) in the first term $Q_1$ of this equation, we get
	\begin{align*}
	Q_1&=\sum_{i=1}^{s_1}g^f(\mathcal{R}(\alpha_1^h,dx_{i}^h)dx_{i}^h,\beta_1^h)\\
	&=\sum_{i=1}^{k_1}g_B((\mathcal{R}_B(\alpha_1,dx_{i})dx_{i})^h,\beta_1^h)\\
	&=Ric_B(\alpha_1,\beta_1)^h.
	\end{align*}
	The second term $Q_2$ is given by
	\begin{align*}
	Q_2&=(f^h)^2\sum_{j=1}^{s_2}g^f(\mathcal{R}(\alpha_1^h,dy_{j}^v)dy_{j}^v,\beta_1^h)\\
	&=-2s_2\Big(\frac{g_B(J_1df,\alpha_1)g_B(J_1df,\beta_1)}{f^2}\Big)^h-s_2\Big(\frac{g_B(\mathcal{D}_{\alpha_1}^BJ_1df,\beta_1)}{f}\Big)^h.
	\end{align*}
	Thus the proof of (a) is follows. Now, prove of (b) is given by
	\begin{align*}
	Ric(\alpha_1^h,\beta_2^v)&=\sum_{i=1}^{s_1}g^f(\mathcal{R}(\alpha_1^h,dx_{i}^h)dx_{i}^h,\beta_2^v)+(f^h)^2\sum_{j=1}^{s_2}g^f(\mathcal{R}(\alpha_1^h,dy_{j}^v)dy_{j}^v,\beta_2^v).
	\end{align*}
	Using Proposition 3.2 and system of equations (2.10) in this equation we obtain
	\begin{align*}
	Ric(\alpha_1^h,\beta_2^v)&=(f^h)^2\sum_{j=1}^{s_2}g^f(\mathcal{R}(\alpha_1^h,dy_{j}^v)dy_{j}^v,\beta_2^v).
	\end{align*}
	This provides the result. Proof of (c) is similar to the (a) and (b).
\end{proof}
\begin{corollary}
	Let $(M=B\times_{f}F,g^f,\Pi)$ be a pseudo-Riemannian Poisson warped product space and $f$ is a Casimir function on $B$ then for any $\alpha_1,\beta_1\in\Gamma(T^*B)$ and $\alpha_2,\beta_2\in\Gamma(T^*F)$, we have
	\begin{align*}
	\textbf{(a).}\:Ric(\alpha_1^{h},\beta_1^{h})&=Ric_B(\alpha_1,\beta_1)^h,\\
	\textbf{(b).}\:Ric(\alpha_1^{h},\beta_2^{v})&=0,\\
	\textbf{(c).}\:Ric(\alpha_2^{v},\beta_2^{v})&=Ric_F(\alpha_2,\beta_2)^v.
	\end{align*}
\end{corollary}
\begin{proof}
As, $f$ is Casimir function $f$ on $B$ if and only if $J_1df=0$. Thus
 after using this in Proposition 3.3, provides the results.
\end{proof}
\subsection{Constant scalar curvature}
In the following proposition and corollary, we have to calculate the scalar curvature on contravariant warped product space and pseudo-Riemannian Poisson warped product respectively.
\begin{proposition}
	Let $(M=B\times_f F, g^f)$ be a contravariant warped product space then
	\begin{equation}
	S=S_B^h+(f^h)^2S_F^v-s_2\Big(\frac{(s_2+3)}{f^2}||J_1df||_B^{2}+\frac{2}{f}\Delta^{\mathcal{D}^B}(f)\Big)^h.
	\end{equation}
\end{proposition}
	\begin{proof}
	As $\{dx_1^{h},...,dx_{s_1}^{h},f^hdy_1^{v},...,f^hdy_{s_2}^{v}\}$ is a local $g^f$-orthonormal basis then
	\begin{align*}
	S&=\sum_{i=1}^{s_1}Ric(dx_{i}^h,dx_{i}^h)+\sum_{j=1}^{s_2}Ric(f^hdy_{j}^v,f^hdy_{j}^v)\\
	&=\sum_{i=1}^{s_1}Ric(dx_{i}^h,dx_{i}^h)+(f^h)^2\sum_{j=1}^{s_2}Ric(dy_{j}^v,dy_{j}^v).
	\end{align*}
Using Proposition 3.3 in the first term $P_1$ of this equation provides
\begin{align*}
    P_1&=\sum_{i=1}^{s_1}Ric(dx_{i}^h,dx_{i}^h)\\
    &=\sum_{i=1}^{s_1}Ric_B(dx_i,dx_i)^h-\frac{s_2}{(f^h)^2}\sum_{i=1}^{s_1}\big[2A(dx_i,dx_i)+fg_B(\mathcal{D}_{dx_i}^B(J_1df),dx_i)\big]^h\\
    &=S_B+\frac{s_2}{(f^h)^2}\big[2||J_1df||_B^{2}+f\Delta^{\mathcal{D}^B}(f)\big]^h.
\end{align*}
The second term $P_2$ is given by
\begin{align*}
P_2&=(f^h)^2\sum_{j=1}^{s_2}Ric(dy_{j}^v,dy_{j}^v)\\
&=(f^h)^2\sum_{j=1}^{s_2}Ric_B(dy_j,dy_j)^h-\Big(\frac{(s_2+1)||J_1df||_B^{2}}{f^2}+\frac{\Delta^{\mathcal{D}^B}(f)}{f}\Big)^h\sum_{i=1}^{s_2}g_F(dy_j,dy_j)^v\\
&=(f^h)^2S_F-s_2\Big(\frac{(s_2+1)||J_1df||_B^{2}}{f^2}+\frac{\Delta^{\mathcal{D}^B}(f)}{f}\Big)^h.
\end{align*}
The result follows.
\end{proof}
\begin{corollary}
	Let $(M=B\times_{f}F,g^f,\Pi)$ be a pseudo-Riemannian Poisson warped product space and $f$ is a Casimir function on $B$, then
	\begin{equation}
	S=S_B^h+(f^h)^2S_F^v.
	\end{equation}
	\begin{proof}
The equation $(3.11)$ follows after using the hypothesis of Casimir function $f$ in Proposition 3.5.	
\end{proof}
\end{corollary}
Next, we will try to provide the answer to the following question, under the assumption that $S_F(y)=\mu$(constant) on $F$, can we find a warping function $f>0$ on $B$ such that the contravariant warped metric $g^f$ has constant scalar curvature $S(x,y)=\mu_1$ on pseudo-Riemannian Poisson warped space $(M=B\times_{f}F,g^f,\Pi)$?
\begin{theorem}
Let $(M=B\times_{f}F,g^f,\Pi)$ be a pseudo-Riemannian Poisson warped product space with $dim(F)=s_2>1$ and $f$ is Casimir function on $B$. If the fiber space has constant scalar curvature $\mu\neq0$. Then $g^f$ admits the following warping function $f$ for which M has a constant scalar curvature $\mu_1$,\\
$\textbf{(a).}$ $\mu_1>S_B,\quad f=\sqrt{\frac{\mu_1-S_B}{\mu}},\quad  \mu>0$,\\
$\textbf{(b).}$ $\mu_1<S_B,\quad f=\sqrt{\frac{\mu_1-S_B}{\mu}},\quad  \mu<0$,\\
$\textbf{(c).}$ If $\mu_1=S_B,\:$ does not exist warping function.
\end{theorem}
\begin{proof}
If $S(x,y)=\mu_1$, equation $(3.10)$ is pullback by $\pi$ of the following equation
$$\mu_1=S_B+f^2\mu-s_2\Big(\frac{(s_2+3)}{f^2}||J_1df||_B^{2}+\frac{2}{f}\Delta^{\mathcal{D}^B}(f)\Big),$$
or equivalently,
\begin{equation}
\frac{1}{2s_2}(S_B-\mu_1)f+\frac{f^3}{2s_2}\mu-\frac{(s_2+3)}{2f}||J_1df||_B^{2}-\Delta^{\mathcal{D}^B}(f)=0.
\end{equation}
After using the hypothesis of Casimir function $f$ in equation $(3.12),$ we have
$$S_B-\mu_1+f^2\mu=0.$$
This provides proof of (a), (b) and (c).
\end{proof}
\section{\textbf{Characterization of warping function on Einstein Poisson warped product space}}
Let $\mathcal{D}$ is the contravariant Levi-Civita connection associated with the pair $(\Pi,g^f)$ (where $\Pi=\Pi_B+\Pi_F$) on contravariant warped product space $(M=B\times_f F,g^f)$. In the following two theorems, we will discuss Einstein criteria on contravariant warped product space  $(M=B\times_f F,g^f)$ and Riemannian Poisson warped product space $(M=B\times_{f}F,g^f,\Pi)$.
\begin{theorem} Let $\mathcal{D}$ is a contravariant Levi-Civita connection associated with the pair $(\Pi,g^f)$ on contravariant warped product space $(M=B\times_f F,g^f)$. Then $M$ is Einstein with $Ric=\lambda g^f$ if and only if the following conditions are satisfied,\\
	$\textbf{(a).}$ $Ric_B=\lambda g_B+\frac{s_2}{f^2}\big[2A-fH_{\Pi_B}^f\big],$ where $A(.,.)=g_B(J_1df,.)g_B(J_1df,.)$,\\
	$\textbf{(b).}$ $(F,g_F)$ is Einstein with $Ric_F=\tilde{\lambda} g_F$,\\
	$\textbf{(c).}$ $\tilde{\lambda}=\frac{1}{f^4}\big[\lambda f^2+(s_2+1)||J_1df||_B^{2}+f\Delta^{\mathcal{D}^B}(f)\big]$.
\end{theorem}
\begin{proof}
	 $Ric(\alpha,\beta)=\lambda g^f(\alpha,\beta)$ if and only if $Ric(\alpha_1^h,\beta_1^h)=\lambda g_B(\alpha_1,\beta_1)^h$,\\ $Ric(\alpha_2^v,\beta_2^v)=\frac{\lambda}{(f^h)^2}g_B(\alpha_2,\beta_2)^v$ and $Ric(\alpha_1^h,\beta_2^v)=0,$\\  for any $\alpha=\alpha_1^h+\alpha_2^v$ and $\beta=\beta_1^h+\beta_2^v$, where $\alpha_1,\beta_1\in\Gamma(T^*B)$, $\alpha_2,\beta_2\in\Gamma(T^*F)$. After using Proposition 3.3 and properties of pullback maps $\pi$ and $\sigma$ in the above equalities provide (a), (b) and (c).
\end{proof}
\begin{remark}
	With the notations of Theorem 4.1, the triplet $(M=B\times_{f}F,g^f,\Pi)$ is called contravariant Einstein warped product space.
\end{remark}
Under the assumption of Theorem 2.3, we will provide necessary and sufficient conditions for Riemannian Poisson warped product space $(M=B\times_{f}F,g^f,\Pi)$ to be an Einstein.
\begin{theorem}
	Let $(M=B\times_{f}F,g^f,\Pi)$ be a Riemannian Poisson warped product space and $f$ is a Casimir function on $B$. Then $M$ is Einstein with $Ric=\lambda g$ if and only if $B$ and $F$ are Einstein with $Ric_B=\lambda g_B$ and $Ric_F=\mu g_F$ respectively
	where $\tilde{\lambda}=\frac{\lambda}{f^2}$.
\end{theorem}
\begin{proof}
	Note that $f$ is Casimir function if and only if $J_1df=0$, using this hypothesis in Theorem $(4.1)$ result is proved.
\end{proof}
\begin{remark}
	With the notations of Theorem 4.3, $(M=B\times_{f}F,g^f,\Pi)$ is called Einstein Poisson warped product space.
\end{remark}
Now, we will discuss an example that verifies the above corollary.
\begin{example}
Let $M=I\times_{f}F$ be a generalized Robertson-Walker space-times with metric $\tilde{g}^f=\tilde{g_{I}}+f(t)^2\tilde{g_{F}}$, where $I$ be an interval and $\tilde{g_{I}}=-dt^2$. We denote the cometric of the metric $\tilde{g_{I}}$ by $g_I$ and
$g_I(dt,dt)=-1$. Thus $g^f=g_I^h+\frac{1}{f(t)^2}g_F^v$ is a cometric of the warped metric $\tilde{g}^f$. Let $\Pi$ and $f$ are the bivector and Casimir function on $M$ and $I$ respectively. Then $(M,g^f,\Pi)$ is a pseudo-Riemannian Poisson warped product space.\par Now let $M$ is an Einstein manifolds with scalar $\lambda$, provides
\begin{align}
Ric(dt^h,dt^h)&=\lambda g_I(dt^h,dt^h)=\lambda g_I(dt,dt)^h,\\
Ric(\alpha_2^{v},\beta_2^{v})&=\frac{\lambda}{f(t)^2}g_F(\alpha_2,\beta_2)^v.
\end{align}
Next from Corollary 3.4, we have
\begin{align}
Ric(dt^h,dt^h)&=Ric_I(dt,dt)^h,\\
Ric(\alpha_2^{v},\beta_2^{v})&=Ric_F(\alpha_2,\beta_2)^v.
\end{align}
From equation (4.1) and (4.3), we obtain
\begin{equation}
Ric_I(dt,dt)=\lambda g_I(dt,dt).
\end{equation}
Similarly, from equation (4.2) and (4.4), we have
\begin{equation}
Ric_F(\alpha_2,\beta_2)=\frac{\lambda}{f(t)^2}g_F(\alpha_2,\beta_2).
\end{equation}
 After taking trace of (4.5) implies that $\lambda=0$.
Thus from equations (4.5), (4.6) and by the hypothesis of $\lambda$, we conclude that $M$ is Ricci-flat if and only if $I$ and $F$ are.
\end{example}
\subsection{Einstein Poisson warped product space with 1-dim base (dim B=1)}
In the following two theorems, we will provide the information of warping function for Einstein Poisson warped product space $(M=B\times_{f}F,g^f,\Pi)$ when $dim(B)=s_1=1$ and $f$ is a Casimir function on $B$.
\begin{theorem}
	Let $(M=B\times_{f}F,g^f,\Pi)$ be an Einstein Poisson warped product space with $dim(B)=s_1=1$ and $s_2>1$, $f$ is Casimir function on $B$ and $\hat{\lambda}>0$. Then\\
	$\textbf{(a).}$ If $\lambda>0$, $f=\sqrt{\frac{\lambda}{\hat{\lambda}}}$.\\
	$\textbf{(b).}$ If $\lambda<0$, does not exist warping function.
\end{theorem}
\begin{proof}
We may replace the two conditions $(a)$ and $(b)$ of Theorem 4.1, by the unique equation
\begin{align}
Ric_B-\frac{s_2}{f^2}\big[2A-fH_{\Pi_B}^f\big]&=\frac{1}{2}\Big[S_B-\frac{s_2(s_2+3)}{f^2}||J_1df||_B^{2}
-\frac{2s_2}{f}\Delta^{\mathcal{D}^B}(f)\nonumber\\
&+s_2f^2\hat{\lambda}-(s_1+s_2-2)\lambda\Big]g_B.
\end{align}
If $f$ is Casimir function on $B$ then equation (4.7), provides
	\begin{equation}
	Ric_B=\frac{1}{2}\big[S_B+s_2f^2\hat{\lambda}-(s_1+s_2-2)\lambda\big]g_B.
	\end{equation}
As dimension of $B$ is one therefore $S_B=\lambda$, thus after taking the trace of $(4.8)$ and using this, we have
\begin{equation}
\hat{\lambda}f^2=\lambda.
\end{equation}
This provides proof of both $(a)$ and $(b)$.
\end{proof}
\begin{theorem}
	Let $(M=B\times_{f}F,\Pi,g^f)$ be a Einstein Poisson warped product space with $dim(B)=s_1=1$ and $s_2>1$, $f$ is Casimir function on $B$ and $\hat{\lambda}<0$. Then\\
	$\textbf{(a).}$ If $\lambda>0$, does not exist warping function. \\
	$\textbf{(b).}$ If $\lambda<0$, $f=\sqrt{\frac{\lambda}{\hat{\lambda}}}$.
\end{theorem}
\begin{proof}
	Proof of both follows from (4.9).
\end{proof}
\subsection{Einstein Poisson warped product space with dim B$\geq2$} In the following two theorems, we will provide the information of warping function for Einstein Poisson warped product space $(M=B\times_{f}F,g^f,\Pi)$ when $dim(B)=s_1\geq 2$ and $f$ is a Casimir function on $B$.
\begin{theorem}
	Let $(M=B\times_{f}F,g^f,\Pi)$ be an Einstein Poisson warped product space with $dim(B)=s_1\geq2$, $f$ is Casimir function on $B$ and $\hat{\lambda}>0$. Then\\
	$\textbf{(a).}$ If $\lambda>0$, $f=\sqrt{\frac{\lambda}{\hat{\lambda}}}$.\\
	$\textbf{(b).}$ If $\lambda<0$, does not exist warping function.
\end{theorem}
\begin{proof}
	As $dim(B)=s_1\geq2$ therefore $S_B=\lambda s_1$, thus after taking the trace of $(4.8)$ and using this, we have
	\begin{equation}
	\hat{\lambda}f^2=\lambda
	\end{equation}
This provides the proof of both both (a) and (b).
\end{proof}
\begin{theorem}
	Let $(M=B\times_{f}F,g^f,\Pi)$ be an Einstein Poisson warped product space with $dim(B)=s_1\geq2$, $f$ is Casimir function on $B$ and $\hat{\lambda}<0$. Then\\
	$\textbf{(a).}$ If $\lambda>0$, does not exist warping function. \\
	$\textbf{(b).}$ If $\lambda<0$, $f=\sqrt{\frac{\lambda}{\hat{\lambda}}}$.
\end{theorem}
\begin{proof}
	Proof of both (a) and (b) follows from (4.10).
\end{proof}

\end{document}